\documentclass{amsart}

\usepackage{hyperref}
\usepackage{tikz}
\usepackage{lipsum}

\newtheorem{theorem}{Theorem}[section]
\newtheorem{lemma}[theorem]{Lemma}

\theoremstyle{definition}

\theoremstyle{remark}

\DeclareMathOperator{\Gl}{GL}
\DeclareMathOperator{\Gal}{Gal}
\DeclareMathOperator{\Aut}{Aut}

\numberwithin{equation}{section}

\address{Department of Mathematics, Faculty of Science, University of Zagreb\\
	Bijeni\v cka cesta 30\\
	10000 Zagreb\\
	Croatia}
\email{tguzvic@math.hr }

\begin{document}

\title[]{Torsion of elliptic curves with rational $j$-invariant defined over number fields of prime degree}


\author{Tomislav Gu\v{z}vi\'c}
\thanks{The author gratefully acknowledges support
from the QuantiXLie Center of Excellence, a project co-financed by the Croatian Government and European Union through the
European Regional Development Fund - the Competitiveness and Cohesion Operational Programme (Grant KK.01.1.1.01.0004) and
by the Croatian Science Foundation under the project no. IP-2018-01-1313.}

\date{}

\dedicatory{}

\begin{abstract}
Let $[K:\mathbb{Q}]=p$ be a prime number and let $E/K$ be an elliptic curve with $j(E) \in \mathbb{Q}$. We determine the all possibilities for $E(K)_{tors}$. We obtain these results by studying Galois representations of $E$ and of it's quadratic twists.
\end{abstract}

\maketitle
\section{Introduction}
Let $K$ be a number field such that $[K:\mathbb{Q}]=d$ and let $E/K$ be an elliptic curve. The set of $K$-rational points, $E(K)$ can be given a group structure. A theorem of Mazur shows that $E(K)$ is a finitely generated abelian group. Therefore this group can be decomposed as $E(K)=E(K)_{tors} \oplus \mathbb{Z}^{r}$, $r \ge 0$. It is known that $E(K)_{tors}$ is of the form $C_{m} \oplus C_{n}$ for two positive integers $m,n$ such that $m$ divides $n$, where $C_m$ and $C_n$ denote cyclic groups of order $m$ and $n$, respectively.
\\
\\
One of the goals in the theory of elliptic curves is the classification of torsion groups of elliptic curves defined over various fields. We will now briefly describe results related to this paper.
\\
Let $d$ be a positive integer.
\begin{enumerate}
    \item Let $\Phi(d)$ be the set of possible isomorphism classes of groups $E(K)_{tors}$, where $K$ runs through all number fields $K$ of degree $d$ and $E$ runs through all elliptic curves over $K$. In \cite{21}, Merel proved that $\Phi(d)$ is finite for all positive integers $d$. The set $\Phi(1)$ was determined by Mazur (\cite{11}) and $\Phi(2)$ was determined by Kenku, Momose and Kamienny (\cite{20}, \cite{13}). Derickx, Etropolski, Hoeij, Morrow and Zureick-Brown have determined $\Phi(3)$ but the result is (at the time of writing this paper) unpublished.
    
    \item Let $\Phi^{\mathrm{CM}}(d)$ be the set of possible isomorphism classes of groups $E(K)_{tors}$, where $K$ runs through all number fields $K$ of degree $d$ and $E$ runs through all elliptic curves with $\mathrm{CM}$ over $K$. $\Phi^{\mathrm{CM}}(1)$ has been determined by Olson in \cite{32} and $\Phi^{\mathrm{CM}}(d)$ for $d=3,4$ by Zimmer and his collaborators.
    The sets $\Phi^{\mathrm{CM}}(d)$, for $4\le d \le 13$ have been determined by Clark, Corn, Rice and Stankiewicz in \cite{30}. Bourdon and Pollack have determined torsion groups of $\mathrm{CM}$ elliptic curves over odd degree number fields in \cite{39}.

    \item Let $\Phi_{\mathbb{Q}}(d) \subseteq \Phi(d)$ be the set of possible isomorphism classes of groups $E(K)_{tors}$, where $K$ runs through all number fields $K$ of degree $d$ and $E$ runs through all elliptic curves defined over $\mathbb{Q}$. For $d=2,3$, the sets $\Phi_{\mathbb{Q}}(d)$ have been determined by Najman (\cite{7}) while $\Phi_{\mathbb{Q}}(4)$ has been determined by Chou (\cite{22}) and Gonz\'alez-Jim\'enez, Najman (\cite{2}). The set $\Phi_{\mathbb{Q}}(5)$ has been determined by Gonz\'alez-Jim\'enez in \cite{24}. Gonz\'alez-Jim\'enez and Najman have also proved that $\Phi_{\mathbb{Q}}(p)=\Phi(1)$ for prime $p \ge 7$ in \cite{2}. In \cite{41}, the author has partially determined $\Phi_{\mathbb{Q}}(6)$.

    \item Let $\Phi_{j \in \mathbb{Q}}(d)$ be the set of possible isomorphism classes of groups $E(K)_{tors}$, where $K$ runs through all number fields $K$ of degree $d$ and $E$ runs through all elliptic curves defined over $K$ such that $j(E) \in \mathbb{Q}$.

    \item Let $R_{\mathbb{Q}}(d)$ be the set of all primes $p$ such that there exists a number field $K$ of degree $d$, an elliptic curve $E/\mathbb{Q}$ such that there exists a point of order $p$ on $E(K)$. This set has been determined by Gonz\'alez-Jim\'enez and Najman (\cite{2}).
\end{enumerate}
We now state the main results of this paper in which we classify the sets $\Phi_{j \in \mathbb{Q}}(p)$, where $p$ is a prime number.

\begin{theorem}
Let $K$ be a number field such that $[K:\mathbb{Q}]=p \ge 7$ is prime and let $E/K$ be an elliptic curve with rational $j$-invariant. If $E(K)$ contains a point of order $n > 1$, then $C_{n} \in \Phi(1)$.
\end{theorem}

\begin{theorem}
Let $K$ be a number field such that $[K:\mathbb{Q}]=5$ and let $E/K$ be an elliptic curve with rational $j$-invariant. If $E(K)$ contains a point of order $n > 1$, then $C_{n} \in \Phi_{\mathbb{Q}}(5)$.
\end{theorem}

\begin{theorem}
Let $K$ be a number field such that $[K:\mathbb{Q}]=3$ and let $E/K$ be an elliptic curve with rational $j$-invariant. If $E(K)$ contains a point of order $n > 1$, then $C_{n} \in \Phi_{\mathbb{Q}}(3)$.
\end{theorem}

\begin{theorem}
Let $K$ be a number field such that $[K:\mathbb{Q}]=2$ and let $E/K$ be an elliptic curve with rational $j$-invariant. If $E(K)$ contains a point of order $n > 1$, then $C_{n} \in \Phi_{\mathbb{Q}}(2) \cup \{ C_{13} \}$.
\end{theorem}
Obviously we have $\Phi_{\mathbb{Q}}(d) \subseteq \Phi_{j \in \mathbb{Q}}(d)$. If $E/K$ is an elliptic curve such that $j(E) \in \mathbb{Q} \setminus \{ 0, 1728 \}$, take $E'/\mathbb{Q}$ to be any elliptic curve such that $j(E')=j(E)$. Then $E$ and $E'$ are either isomorphic over $K$ or over some quadratic extension $L$ of $K$. Assume that $C_{m} \oplus C_{n} \subseteq E(K)$. This implies that $C_{m} \oplus C_{n} \subseteq E'(L)$, so $C_{m} \oplus C_{n}$ is a subgroup of one of the groups appearing in $\Phi_{\mathbb{Q}}(2d)$.
\\
\\
Let us consider the case when $d$ is odd. Assume that $C_{m} \oplus C_{n}$, where $m$ divides $n$ is contained in $E(K)$. By the properties of Weil pairing we have $\mathbb{Q}(\zeta_{m}) \subseteq K$. If $m \ge 3$, $[\mathbb{Q}(\zeta_m):\mathbb{Q}]$ is even so $\mathbb{Q}(\zeta_{m})$ can't be a subfield of $K$. Therefore, when trying to classify $\Phi_{j \in \mathbb{Q}}(p)$ we shall consider only groups of the form $C_n$ and $C_2 \oplus C_{2n}$. 
\\
\\
We now describe the general strategy used to solve this problem. Let $K$ be a number field of degree $p$ and $E(K)$ be an elliptic curve with $j(E) \in \mathbb{Q} \setminus \{ 0, 1728 \}$ and let $P_n \in E(K)$ be a point of order $n$. If $E$ is a base change of an elliptic curve defined over $\mathbb{Q}$, we are done, because $E(K)_{tors} \in \Phi_{\mathbb{Q}}(p)$. Otherwise, we take any elliptic curve $E'$ defined over $\mathbb{Q}$ such that $j(E')=j(E)$. With $L$ we shall denote (unless otherwise stated) a quadratic extension of $K$ such that $E$ and $E'$ are isomorphic over $L$, so they are quadratic twists of each other. On the one hand, if $\{ P_n, Q_n \}$ is a basis for $E[n]$, the image of $\rho_{E,n}$ is a subgroup of $B_{0}$, where \[ B_{0}:=\bigg\{ \begin{bmatrix}
 1 && * \\
0 && *
\end{bmatrix} \bigg\} .\] Since $\rho_{E,n} \sim \chi \cdot \rho_{E',n}$, where $\chi$ is a quadratic character, we can obtain some information about the mod $n$ Galois representation of $E'$. On the other hand, since $E$ and $E'$ are isomorphic over $L$, there exists a point $P'_n \in E'(L)$ of order $n$. For each prime divisor $q$ of $n$, let $P'_q \in E'(L)$ be a point of order $q$. Obviously $[\mathbb{Q}(P'_q):\mathbb{Q}]$ is a divisor of $[L:\mathbb{Q}]$. Using the results of \cite{2}, we can check the possible values of $[\mathbb{Q}(P'_q):\mathbb{Q}]$. Often it turns out that $E'$ has a rational $q$-isogeny.

\section{Notation and auxiliary results}

Let $E/F$ be an elliptic curve defined over a number field $F$. There exists an $F$-rational cyclic isogeny $\phi : E \to E'$ of degree $n$ if and only if $\ker\phi$ is a $Gal(\Bar{F}/F)$-
invariant cyclic group of order $n$; in this case we say that $E$ has an $F$-rational $n$-isogeny.
When $F=\mathbb{Q}$, possible degrees $n$ of elliptic curves over $\mathbb{Q}$ are known by the following theorem.

\begin{theorem}[Mazur, \cite{11}]
\label{Theorem 2.1}
Let $E/\mathbb{Q}$ be an elliptic curve. Then

\[ E(\mathbb{Q})_{tors} \cong
\begin{cases} 
      C_m, & m=1,...,10, 12, \\
      C_{2} \oplus C_{2m}, & m=1,...,4.
   \end{cases} \]
\end{theorem}

\begin{theorem}[Kenku, Momose, \cite{20}, Kamienny \cite{13}]
\label{Theorem 2.2}
Let $E/F$ be an elliptic curve over a quadratic number field $F$. Then 
\[ E(F)_{tors} \cong
\begin{cases} 
      C_m, & m=1,...,16,18, \\
      C_{2} \oplus C_{2m}, & m=1,...,6,\\
      C_3 \oplus C_{3m}, & m=1,2, \\
      C_4 \oplus C_4.
   \end{cases} \]
   
\end{theorem}

\begin{theorem}[Mazur \cite{12}, Kenku \cite{14}, \cite{15}, \cite{16}, \cite{17}]
\label{Theorem 2.3}
Let $E/\mathbb{Q}$ be an elliptic curve with a rational $n$-isogeny. Then \[ n \in \{ 1,...,19,21, 25,27, 37, 43, 67, 163 \} .\]
There are infinitely many elliptic curves (up to $\bar{\mathbb{Q}}$-isomorphism) with a rational $n$-isogeny over $\mathbb{Q}$ for \[ n \in \{
{1, . . . , 10, 12, 13, 16, 18, 25} \} \] and only finitely many for all the other $n$.
\end{theorem}

\begin{theorem} \cite[Theorem 2]{7}
\label{Theorem 2.4}
Let $E/\mathbb{Q}$ be an elliptic curve and $F$ a quadratic field. Then 
\[ E(F)_{tors} \cong
\begin{cases} 
      C_m, & m=1,...,10,12,15,16, \\
      C_{2} \oplus C_{2m}, & m=1,...,6,\\
      C_3 \oplus C_{3m}, & m=1,2, \\
      C_4 \oplus C_4.
   \end{cases} \]
\end{theorem}

\begin{theorem} \cite[Theorem 1]{7}
\label{Theorem 2.5}
Let $E/\mathbb{Q}$ be an elliptic curve and $K$ a cubic field. Then 
\[ E(K)_{tors} \cong
\begin{cases} 
      C_m, & m=1,...,10, 12, 13, 14, 18, 21, \\
      C_{2} \oplus C_{2m}, & m=1,...,4,7.
   \end{cases} \]

\end{theorem}

\begin{theorem} \cite{2,22}
\label{Theorem 2.6}
Let $E/\mathbb{Q}$ be an elliptic curve and $K$ a quartic field. Then

\[ E(K)_{tors} \cong
\begin{cases} 
      C_m, & m=1,...,10, 12, 13, 15, 16, 20, 24,\\
      C_{2} \oplus C_{2m}, & m=1,...,6, 8,\\
      C_3 \oplus C_{3m}, & m=1, 2, \\
      C_4 \oplus C_{4m}, & m=1,2, \\
      C_5 \oplus C_5, \\
      C_6 \oplus C_{6}.
   \end{cases} \]
\end{theorem}

\begin{theorem} \cite{38}
\label{Theorem 2.7}
Let $E/\mathbb{Q}$ be an elliptic curve and $K$ a quintic field. Then
\[ E(K)_{tors} \cong
\begin{cases} 
      C_m, & m=1,...,12, 25, \\
      C_{2} \oplus C_{2m}, & m=1,...,4.
   \end{cases} \]
\end{theorem}

\begin{theorem} \cite{41}
\label{Theorem 2.8}
Let $E/\mathbb{Q}$ be an elliptic curve and let $K$ be a sextic number field. Then
\[ E(K)_{tors} \cong
\begin{cases} 
      C_m, & m=1,...,16, 18, 21, 30, m \ne 11,\\
      C_{2} \oplus C_{2m}, & m=1,...,7, 9,\\
      C_3 \oplus C_{3m}, & m=1,...,4, \\
      C_4 \oplus C_{4m}, & m=1,3, \\
      C_6 \oplus C_6, \\
      C_3 \oplus C_{18}.
   \end{cases} \]
   
Additionally, if $G_{E}(2) \neq 2B$, then $E(K)$ is never isomorphic to $C_3 \oplus C_{18}$.
\end{theorem}

\subsection*{Galois representations attached to elliptic curves}
Let $E/\mathbb{Q}$ be an elliptic curve and let $n$ a positive integer. By $E[n]$ we shall denote the $n$-torsion subgroup of $E(\overline{\mathbb{Q}})$. The field $\mathbb{Q}(E[n])$ is the number field obtained by adjoining to $\mathbb{Q}$ all the $x$ and $y$-coordinates of the points of $E[n]$. The absolute Galois group $Gal(\overline{\mathbb{Q}}/\mathbb{Q})$ acts on $E[n]$ by its action on the coordinates of the
points, inducing a mod $n$ Galois representation attached to $E$:
\[ \rho_{E,n}: \Gal(\overline{\mathbb{Q}}/\mathbb{Q}) \to \Aut(E[n]) .\]
After we fix a base for the $n$-torsion, we can identify $\Aut(E[n])$ with $\Gl_{2}(\mathbb{Z}/n\mathbb{Z})$.
This means that we can consider $\rho_{E,n}(\Gal(\overline{\mathbb{Q}}/\mathbb{Q}))$ as a subgroup of $\Gl_{2}(\mathbb{Z}/n\mathbb{Z})$, uniquely determined up to conjugacy. We shall denote $\rho_{E,n}(\Gal(\overline{\mathbb{Q}}/\mathbb{Q}))$ by $G_{E}(n)$. Moreover, since $\mathbb{Q}(E[n])$ is a Galois extension of $\mathbb{Q}$ and $\ker \rho_{E,n}= \Gal(\overline{\mathbb{Q}}/\mathbb{Q}(E[n]))$, by the first isomorphism theorem we have $G_{E}(n) \cong \Gal(\mathbb{Q}(E[n])/\mathbb{Q})$.
\\
Rouse and Zureick-Brown \cite{31} have classified all the possible $2$-adic images of $\rho_{E,2^{\infty}}: \Gal(\overline{\mathbb{Q}}/\mathbb{Q}) \to \Gl_{2}(\mathbb{Z}_2)$, and have given explicitly all the $1208$ possibe images. We will use the same notation as in \cite{31} for the $2$-adic image of a given elliptic curve $E/\mathbb{Q}$. In \cite{5}, Gonz\'alez-Jim\'enez and Lozano-Robledo have determined for each possible image the degree of the field of definition of any $2$-subgroup.
From the results of \cite{36} one can see if a given $2$-subgroup is defined over a number field of given degree $d$. 

\subsection*{Division polynomial method}
 $E/\mathbb{Q}$ be an elliptic curve and $n$ a positive integer. We denote by $\psi_{E,n}$ the $n$-th division polynomial of $E$ (see \cite[Section 3.2]{33}). If $n$ is odd, then the roots of $\psi_{E, n}$ are precisely the $x$-coordinates of points $P \in E[n]$. Similarly, if $n$ is even, then the roots of $\psi_{E, n}/\psi_{E,2}$ are precisely the $x$-coordinates of points $P \in E[n]\setminus E[2]$. Let $f_{E,n}$ denote the corresponding primitive $n$-division polynomial associated to $E$, i.e. it's roots are the $x$-coordinates of points $P$ on $E(\overline{\mathbb{Q}})$ of exact order $n$. We briefly describe the construction of $f_{E,n}$. If $n=p$ is prime, then $f_{E,p}=\psi_{E,p}$. For an arbitrary $n$, we have \[  f_{E,n}:=\frac{\psi_{E,n}}{\displaystyle \prod_{d|n, d \ne n}f_{E,d}}.\] Note that if $E^{d}/\mathbb{Q}$ is a quadratic twist of $E/\mathbb{Q}$, then $\psi_{E,n}=\alpha \psi_{E^{d},n}$ and $f_{E,n}=\beta f_{E^d,n}$, for some rational constants $\alpha$, $\beta$. Consider the following problem:
\\
\\
\textit{Given a rational number $j$ and K a number field of degree $d$, does there exist an elliptic curve $E/\mathbb{Q}$ such that $j=j(E)$ can $E(K)$ contains a point $P$ of exact order $n$?}
\\
\\
 Let $E_{0}/\mathbb{Q}$ be any elliptic curve with $j=j(E_{0})$. In \texttt{Magma} \cite{35}, we compute the primitive division polynomial $f_{E_0,n}$. Since every elliptic curve $E/\mathbb{Q}$ with $j(E)=j$ is a quadratic twist of $E_0$, we have $f_{E_0,n}=\beta f_{E,n}$, for some rational number $\beta$. Next, we factor $f_{E_0,n}$ over $\mathbb{Q}[x]$. Let $d'$ denote the degree of the smallest irreducible factor $f$ of $f_{E_0,n}$ and let $x_0$ be a root of $f$. If $d' >d$, then $[\mathbb{Q}(P):\mathbb{Q}] \ge [\mathbb{Q}(x_0):\mathbb{Q}]=d' > d=[K:\mathbb{Q}]$ and so a point $P$ of exact order $n$ on $E(\overline{\mathbb{Q}})$ can't be defined over $K$.
\\
\\
Specific elliptic curves mentioned in this paper will be referred to by their $\mathrm{LMFDB}$ label and a link to
the corresponding $\mathrm{LMFDB}$ page \cite{18} will be included for the ease of the reader. Conjugacy classes of subgroups
of $\Gl_{2}(\mathbb{Z}/p\mathbb{Z})$ will be referred to by the labels introduced by Sutherland in \cite{19}.

\section{Classification of $\Phi_{j \in \mathbb{Q}}(p)$}

\begin{lemma}
\label{Lemma 3.1}
Let $p \ge 7$ be a prime number. Then $R_{\mathbb{Q}}(2p) = \{ 2, 3, 5, 7 \}$. Moreover, we have $R_{\mathbb{Q}}(10)= \{ 2, 3, 5, 7, 11 \}$ and $R_{\mathbb{Q}}(6)= \{ 2, 3, 5, 7, 13 \}$.
\end{lemma}

\begin{proof}
The claim will follow easily by \cite[Corollary 6.1.]{2}. We briefly sketch the proof.
Let $p \ge 7$ and $q \ge 23$, $q \neq 37, 43, 67, 163$ be a prime numbers and assume that $q \in R_{\mathbb{Q}}(2p)$. We have that $2(q-1)|2p$ or $\frac{q^2-1}{3}|2p$. If $2(q-1)|2p$, we have $q \in \{ 2, p+1 \}$, which is impossible. If $\frac{q^2-1}{3}|2p$, it follows that $(q-1)(q+1)|6p$. If $q-1 \ge p$, then  $(q-1)(q+1) > p^2 > 6p$, a contradiction. Therefore $q-1 \in \{ 1, 2, 3, 6 \}$, so $q \le 7$ which is impossible since $q \ge 23$. It remains to check that the claim holds for $q \in \{ 11, 13, 17, 19, 37, 43, 67, 163 \}$ which is trivial to do.
\\
The claims $R_{\mathbb{Q}}(10)= \{ 2, 3, 5, 7, 11 \}$ and $R_{\mathbb{Q}}(6)= \{ 2, 3, 5, 7, 13 \}$ also easily follow from \cite[Corollary 6.1.]{2}.
\end{proof}

\begin{theorem}
Let $K$ be a number field of prime degree $p$ and let $E/K$ be a $\mathrm{CM}$ elliptic curve with $j(E) \in \mathbb{Q}$. Then we have $E(K)_{tors} \in \Phi_{\mathbb{Q}}(p)$.
\end{theorem}

\begin{proof}
The claim follows easily by \cite[Theorem 1.2.]{39}, \cite[Theorem 1.5.]{40} and Lemma \ref{Lemma 3.1}.
\end{proof}

Therefore, from now on we shall assume that elliptic curves we're dealing with do not have $\mathrm{CM}$.

\begin{lemma}
\label{Lemma 3.2}
Let $E/\mathbb{Q}$ be an elliptic curve and $p \ge 3$ a prime and $[F:\mathbb{Q}]=2$. Assume that $E(F)[p] \supseteq C_p$, but $E(\mathbb{Q})[p]={O}$. Then there there exists quadratic twist $E'/\mathbb{Q}$ of $E/\mathbb{Q}$ such that $E'(\mathbb{Q})[p]=C_p$.
\end{lemma}

\begin{proof}
Since $F=\mathbb{Q}(\sqrt{d})$, put $E':=(E)^{d}$. $E'$ and $E$ are isomorphic over $F$ but not over $\mathbb{Q}$. Since $C_p \subseteq E(F)[p]=E(\mathbb{Q})[p] \oplus E'(\mathbb{Q})[p]$ and $E(\mathbb{Q})[p]={O}$ it follows that $C_p \subseteq E'(\mathbb{Q})[p]$. Obviously equality must hold, because of Mazur's theorem.
\end{proof}

\begin{lemma}
\label{Lemma 3.3}
Let $K$ be a number field such that $[K:\mathbb{Q}]$ is odd, $m \ne 2$ an integer and $E/K$ an elliptic curve with $j(E) \in \mathbb{Q}$ such that $C_m \subseteq E(K)$. If $E'/\mathbb{Q}$ is an elliptic such that $j(E)=j(E')$ and $\mathbb{Q}(E'[m]) \cap K =\mathbb{Q}$, then $G_{E'}(m)$ is conjugate to a subgroup of $B$, where 
\[ B:=\bigg\{ \begin{bmatrix}
\pm 1 && * \\
0 && *
\end{bmatrix} \bigg\} \subseteq \Gl_{2}(\mathbb{Z}/m\mathbb{Z}).\]
\end{lemma}
\begin{proof}
By \cite[Corollary 5.25.]{19} we see that $\Gal(K\mathbb{Q}(E'[m])/K) \le B$ (up to conjugacy). Since $\mathbb{Q}(E'[m]) \cap K =\mathbb{Q}$ we have $\Gal(K\mathbb{Q}(E'[m])/K) \cong \Gal(\mathbb{Q}(E[m])/\mathbb{Q})$ by a basic Galois theory argument. Therefore, $G_{E'}(m)$ is conjugate to a subgroup of $B$.
\end{proof}

\begin{lemma}
\label{Lemma 3.4}
Let $K$ be a number field such that $[K:\mathbb{Q}] \ge 5$ is prime and $E/K$ be an elliptic curve with rational j-invariant and assume that $C_{2^k \cdot 3^l} \subseteq E(K)$. Then $E'$ has a rational $2^k \cdot 3^{l}$-isogeny.
\end{lemma}
\begin{proof}
By \cite[Corollary 2.8.]{39} we have $|\Gl_{2}(\mathbb{Z}/2^{k}\mathbb{Z})|=2^{4k-3} \cdot 3$ and $|\Gl_{2}(\mathbb{Z}/3^{k}\mathbb{Z})|=3^{4k-3} \cdot 2^{4}$. Let $p \in \{ 2, 3 \}$ be a prime number. Since $G_{E}(p) \le \Gl_{2}(\mathbb{Z}/p^{k}\mathbb{Z})$, it follows that $[\mathbb{Q}(E[p]):\mathbb{Q}]=|G_E(p)|$ divides $|\Gl_{2}(\mathbb{Z}/p^{k}\mathbb{Z})| \in \{2^{4k-3} \cdot 3,  3^{4k-3} \cdot 2^{4} \}$ so $\mathbb{Q}(E[p])$ has a trivial intersection with $K$, i.e. $\mathbb{Q}(E[p]) \cap K =\mathbb{Q}$. The claim now follows from Lemma \ref{Lemma 3.3}.
\end{proof}

\begin{lemma}
\label{Lemma 3.5}
Let $K$ be a number field such that $[K:\mathbb{Q}] \neq 5$ is an odd prime and $E/K$ be an elliptic curve with rational $j$-invariant such that $C_5 \subseteq E(K)_{tors}$. Then $E$ is a base change of an elliptic curve defined over $\mathbb{Q}$. If $[K:\mathbb{Q}]=5$ and $C_5 \subseteq E(K)_{tors}$, then $E'$ has a rational $5$-isogeny. 
\end{lemma}

\begin{proof}
Assume that $[K:\mathbb{Q}] \neq 5$ is an odd prime and $E$ is not a base change of an elliptic curve defined over $\mathbb{Q}$. Since $C_5 \subseteq E(K)$, it follows that $C_5 \subseteq E'(L)$. Let $P_5$ be a point of order $5$ on $E'(L)$. We have that $\mathbb{Q}(P_5) \subseteq L$ and so $[\mathbb{Q}(P_5):\mathbb{Q}]$ divides $[L:\mathbb{Q}]=2p$, where $p=[K:\mathbb{Q}]$. By \cite[Table 1]{2}, we see that the only possibilities for $[\mathbb{Q}(P_5):\mathbb{Q}]$ are $1$ and $2$. Now we apply Lemma \ref{Lemma 3.2} to $E'$ to obtain a quadratic twist $E''/\mathbb{Q}$ such that $C_5 \subseteq E''(\mathbb{Q})$. Since $E$ and $E''$ are quadratic twists, they are isomorphic over some quadratic extension $L'$ of $K$ and we have $C_5 \oplus C_5 \subseteq E(K)[5] \oplus E''(K)[5] \cong E''(L')[5]$. The Weil pairing implies that $\mathbb{Q}(\zeta_5) \subseteq L'$ and so $[\mathbb{Q}(\zeta_5):\mathbb{Q}]=4$ divides $[L':\mathbb{Q}]=2p$, which is impossible.
\\ If $[K:\mathbb{Q}]=5$, by applying the same reasoning as in the previous paragraph it can be easily seen that $E'$ must have a rational $5$-isogeny.
\end{proof}

\begin{lemma}
\label{Lemma 3.6}
Let $K$ be a number field such that $[K:\mathbb{Q}] \neq 3, 7$ is prime and $E/K$ be an elliptic curve with rational $j$-invariant such that $C_7 \subseteq E(K)_{tors}$. Then $E$ is a base change of elliptic curve defined over $\mathbb{Q}$. If $[K:\mathbb{Q}]=7$ and $C_7 \subseteq E(K)_{tors}$, then $E'$ has a rational $7$-isogeny. If $[K:\mathbb{Q}]=3$ and $C_7 \subseteq E(K)_{tors}$, then $E'$ has a rational $7$-isogeny  unless $E'$ has $\mathrm{LMFDB}$ label $\href{http://www.lmfdb.org/EllipticCurve/Q/2450ba1/}{\mathrm{2450.y1}}$ or $\href{http://www.lmfdb.org/EllipticCurve/Q/2450bd1/}{\mathrm{2450.z1}}$ (or equivalently, if $G_{E'}(7)$ is conjugate to a group with label $7Ns.2.1.$). 
\end{lemma}

\begin{proof}
The proof is the same as in Lemma \ref{Lemma 3.5}.
\end{proof}

We will classify torsion growth of $E'$, if $E'$ has $\mathrm{LMFDB}$ label $\href{http://www.lmfdb.org/EllipticCurve/Q/2450ba1/}{\mathrm{2450.y1}}$ or $\href{http://www.lmfdb.org/EllipticCurve/Q/2450bd1/}{\mathrm{2450.z1}}$ separately.
\begin{lemma}
Let $E'/\mathbb{Q}$ be a curve with $\mathrm{LMFDB}$ label $\href{http://www.lmfdb.org/EllipticCurve/Q/2450ba1/}{\mathrm{2450.y1}}$ or $\href{http://www.lmfdb.org/EllipticCurve/Q/2450bd1/}{\mathrm{2450.z1}}$ and let $L$ be a number field such that $[L:\mathbb{Q}]=2p$, where $p$ is prime. Then $E'(L)_{tors} \in \{ C_1, C_2,  C_2 \oplus C_2, C_7 \}.$
\end{lemma}

\begin{proof}
Let $q \ne 7$ be a prime and let $E'$ be any of these two curves. Then $G_{E'}(q)=\Gl_{2}(\mathbb{Z}/q\mathbb{Z})$. Therefore, a point $P_q$ of order $q$ on $E'$ satisfies $[\mathbb{Q}(P_q):\mathbb{Q}]=q^2-1$. If $P_q \in E'(L)$, then $[\mathbb{Q}(P_q):\mathbb{Q}]=q^2-1$ would divide $[L:\mathbb{Q}]=2p$. This implies that $q=2$ and $p=3$.
\\
Consider the case $q=7$. If $P_7 \in E'(L)$ is a point of order $7$, by \cite[Table 1]{2} we have $[L:\mathbb{Q}]=6$. Using the algorithm from \cite{1} we see that if $C_7 \subseteq E'(L)$, where $[L:\mathbb{Q}]=12$, then $C_7=E'(K)_{tors}$.
\end{proof}

From now on, assume that $E'$ is not one of these two curves, i.e. if $C_7 \subseteq E(K)$, then $E'$ will have a rational $7$-isogeny by Lemma \ref{Lemma 3.6}. 

\begin{lemma}
\label{Lemma 3.7}
Let $K$ be a number field such that $[K:\mathbb{Q}]$ is odd. Then there does not exist an elliptic curve $E/K$ with rational j-invariant such that $C_{16} \subseteq E(K)$.
\end{lemma}
\begin{proof}
Assume the contrary, that $C_{16} \subseteq E(K)$. It follows that $C_{16} \subseteq E'(L)$ and let $P_{16}$ be a point of order $16$ in $E'(L)$. Since $[\mathbb{Q}(P_{16}):\mathbb{Q}]$ divides $[L:\mathbb{Q}]$ and $4$ does not divide $[L:\mathbb{Q}]$, we have that $[\mathbb{Q}(P_{16}):\mathbb{Q}]$ is not divisible by $4$. By the results of \cite{36}, we see that $G_{E'}(16) \in \{ \href{http://users.wfu.edu/rouseja/2adic/X235l.html}{H_{235l}}, \href{http://users.wfu.edu/rouseja/2adic/X235m.html}{H_{235m}} \}$ and $[\mathbb{Q}(P_{16}):\mathbb{Q}]=2$. In both cases we have $|G_{E'}(16)|=256$. By Lemma \ref{Lemma 3.3}, up to conjugacy we have $G_{E'}(16) \le B$ (where $B$ is a subgroup of upper triangular matrices defined in Lemma \ref{Lemma 3.3}). Since $|B|=256$, the equality holds. But $-I \notin G_{E}(16)$ and $-I \in B$, a contradiction. 
\end{proof}

\begin{lemma}
\label{Lemma 3.8}
Let $K$ be a number field such that $[K:\mathbb{Q}]$ is odd and let $E/K$ be an elliptic curve with rational j-invariant. Then $E(K)$ can't contain $C_2 \oplus C_{12}.$
\end{lemma}

\begin{proof}
Assume that $C_2 \oplus C_{12} \subseteq E(K)$. This implies $C_2 \oplus C_{12} \subseteq E'(L)$ and $[L:\mathbb{Q}]$ is not divisible by $4$. By \cite[Table 1]{2} we see that $E'$ has a rational $3$-isogeny and denote by $\langle P_3 \rangle$ the kernel of such isogeny. Obviously $G_{E'}(4)$ is not surjective, because otherwise a point $P_4$ of order $4$ on $E'$ would be defined over degree $12$ extension of $\mathbb{Q}$, so $12$ would divide $[L:\mathbb{Q}]$. By \cite[Theorem]{4}, we have that $G_{E'}(4) \subseteq H_i$, where $i \in \{ 9, 10, 11, 12, 13 \}$. Since $|H_i|=16$, for $i \in \{ 9, 10, 11, 12, 13 \}$ it follows that $[\mathbb{Q}(E'[4]):\mathbb{Q}]=|G_{E'}(4)|$ is a power of $2$. This implies that $\mathbb{Q}(E'[4]) \cap K=\mathbb{Q}$. Since having a $2$-torsion is twist invariant property and $E(K)[2]=C_2 \oplus C_2$, we have $E'(K)[2]=C_2 \oplus C_2$. We now see that $\mathbb{Q}(E'[2]) \subseteq \mathbb{Q}(E'[4]) \cap K=\mathbb{Q}$, so $\mathbb{Q}(E[2])=\mathbb{Q}$. Since $C_4 \subseteq E(K)$, by lemma \ref{Lemma 3.3}, $E'$ has a rational $4$-isogeny. Let $P_4$ be a kernel of that isogeny and $Q_2$ be such that $\{ 2P_4, Q_2 \}$ is a basis for $E[2]$. The subroups $\langle P_3+P_4 \rangle$ and $\langle Q_2 \rangle$ are kernels of independent $12$ and $2$-isogenies, so $E'$ is isogenous over $\mathbb{Q}$ to a curve $E''/\mathbb{Q}$ with a rational $24$-isogeny, which is impossible by Theorem \ref{Theorem 2.3}.
\end{proof} 

\begin{lemma}
\label{Lemma 3.9}
Let $[K:\mathbb{Q}]=p$ be a prime and let $E/K$ be an elliptic curve with rational j-invariant. Then $E(K)$ can't contain $C_{27}$.
\end{lemma}

\begin{proof}
If $p \ge 5$, by Lemma \ref{Lemma 3.4}, $E'$ has a rational $27$-isogeny, so it has $\mathrm{CM}$ and so $E$ has $\mathrm{CM}$ as well, which contradicts our assumption.
On the other hand, if $p=3$ then $E'(L)$ would contain $C_{27}$, so $C_{27} \in \Phi_{\mathbb{Q}}(6)$, which isn't true by Theorem \ref{Theorem 2.8}.
\end{proof}

\begin{lemma}
\label{Lemma 3.10}
Let $[K:\mathbb{Q}]=p \ge 5$ be a prime and let $E/K$ be an elliptic curve with rational j-invariant. Then $E(K)$ can't contain $C_{18}$.
\end{lemma}

\begin{proof}
Since $E'$ has a rational $2$-isogeny, it has a rational point $P_2$ of order $2$. Let $P_9 \in E'(L)$ be a point of order $9$. We have that $[\mathbb{Q}(P_9):\mathbb{Q}]$ divides $[L:\mathbb{Q}]=2p$. On the other hand, by \cite[Proposition 4.6.]{2} we have that $[\mathbb{Q}(P_9):\mathbb{Q}(3P_9)][\mathbb{Q}(3P_9):\mathbb{Q}]$ divides $18^2$, and since $gcd(18^2,2p)=2$, we have $[\mathbb{Q}(P_9):\mathbb{Q}] \in \{ 1, 2\}$. We conclude that $P_9$ is defined over an at most quadratic extension of $\mathbb{Q}$ and since $P_2 \in E(\mathbb{Q})$, the point $P_2 + P_9$ of order $18$ on $E$ is defined over a quadratic number field, which is imossible since $C_{18} \notin \Phi_{\mathbb{Q}}(2)$, by Theorem \ref{Theorem 2.4}.
\end{proof}

\section{$[K:\mathbb{Q}]=p$, $p \ge 11$}
In this subsection, let $K$ denote a number field such that $[K:\mathbb{Q}]=p$, where $p \ge 11$ is a prime number.
\\
If $P_n \in E(K)$ is a point of order $n$ then $C_n \subseteq E'(L)$. By Lemma \ref{Lemma 3.1} we only need to consider those integers $n$ whose prime factors are contained in $R_{\mathbb{Q}}(2p)=\{ 2, 3, 5, 7 \}$.
\begin{theorem}
Let $K$ be a number field such that $[K:\mathbb{Q}]=p \ge 11$ is prime and let $E/K$ be an elliptic curve with rational $j$-invariant. If $E(K)$ contains a point of order $n > 1$, then $C_{n} \in \Phi(1)$.
\end{theorem}

\begin{proof}
Assume that $E(K)$ contains a point $P_n$ of order $n$. By Lemma \ref{Lemma 3.1}, prime factors of $n$ are contained in $\{ 2, 3, 5, 7 \}$. Therefore, write $n=2^{a}3^{b}5^{c}7^{d}$, where $a,b,c,d \ge 0$. If $(c,d) \neq (0,0)$, then by Lemma \ref{Lemma 3.5} and Lemma \ref{Lemma 3.6} we have that $E$ is a base change of an elliptic curve defined over $\mathbb{Q}$, so the claim holds by \cite[Corollary 7.3.]{2}. Consider now the case when $c=d=0$. By Lemma \ref{Lemma 3.4}, $E'$ has a rational $2^{a}3^{b}$-isogeny. By Theorem \ref{Theorem 2.3} we have \[ n=2^{a}3^{b} \in \{ 1,2,3,4,6,8,9,12,16,18,27 \} .\] Among these values of $n$, we have $C_n \notin \Phi(1)$ only for $n \in \{ 16, 18, 27 \}$. But each of these cases is impossible by Lemma \ref{Lemma 3.7}, Lemma \ref{Lemma 3.10} and Lemma \ref{Lemma 3.9}.
\end{proof}

\section{$[K:\mathbb{Q}]=7$}
In this subsection, let $K$ denote a number field such that $[K:\mathbb{Q}]=7$.
\\
If $P_n \in E(K)$ is a point of order $n$ then $C_n \subseteq E'(L)$. By Lemma \ref{Lemma 3.1} we only need to consider those integers $n$ whose prime factors are contained in $R_{\mathbb{Q}}(14)=\{ 2, 3, 5, 7\}$.

\begin{lemma}
\label{Lemma 5.1}
Let $E/K$ be an elliptic curve with rational j-invariant. Then $E(K)$ can't contain $C_{49}$, $C_{21}$ or $C_{14}$.
\end{lemma}

\begin{proof}
\fbox{$C_{49}$}: This is proven in \cite[Lemma 2.8.]{1}.
\\
\\
\fbox{$C_{21}$}: By Lemma \ref{Lemma 3.4} and Lemma \ref{Lemma 3.6} $E'$ has a rational $3$ and $7$-isogenies, so it has a rational $21$-isogeny so $j(E) \in \{-3^2\cdot5^6/2^3, 3^3\cdot5^3/2, 3^3\cdot5^3\cdot101^3/2^{21}, -3^3\cdot5^3\cdot383^3/2^7 \}$. Using the division polynomial method in \texttt{Magma} \cite{35}, we see that this is impossible.
\\
\\
\fbox{$C_{14}$}: By Lemma \ref{Lemma 3.4} and Lemma \ref{Lemma 3.6} $E'$ has a rational $2$ and $7$-isogenies, so it has a rational $14$-isogeny so $E'$ has $\mathrm{CM}$ which implies that $E$ has $\mathrm{CM}$ aswell.
\end{proof}

\begin{theorem}
Let $K$ be a number field such that $[K:\mathbb{Q}]=7$. If $E(K)$ contains a point of order $n > 1$, then $C_{n} \in \Phi(1)$.
\end{theorem}

\begin{proof}
Assume that $E(K)$ contains a point $P_n$ of order $n$. By Lemma \ref{Lemma 3.1}, the prime factors of $n$ can only be $2, 3, 5, 7$. Therefore, write $n=2^{a}3^{b}5^{c}7^{d}$, where $a,b,c,d \ge 0$. If $c \neq 0$, then by Lemma \ref{Lemma 3.5} $E$ is a base change of elliptic curve defined over $\mathbb{Q}$, so $C_n \in \Phi_{\mathbb{Q}}(7)=\Phi(1)$. Assume that $c=0$. If $d \neq 0$, Lemma \ref{Lemma 5.1} implies that $a=b=0$, so $C_n=C_7 \in \Phi(1)$. It remains to consider the case $c=d=0$. By Lemma \ref{Lemma 3.4}, $E'$ has a rational $2^{a}3^{b}$-isogeny. By Theorem \ref{Theorem 2.3} we have \[ 2^{a}3^{b} \in \{ 1,2,3,4,6,8,9,12,16,18,27 \} .\] Among these values of $n$, we have $C_n \notin \Phi(1)$ only for $n \in \{ 16, 18, 27 \}$. But each of these cases is impossible by Lemma \ref{Lemma 3.7}, Lemma \ref{Lemma 3.10} and Lemma \ref{Lemma 3.9}.
\end{proof}

\section{$[K:\mathbb{Q}]=5$}
In this subsection, let $K$ denote a number field such that $[K:\mathbb{Q}]=5$.
\\
If $P_n \in E(K)$ is a point of order $n$ then $C_n \subseteq E'(L)$. By Lemma \ref{Lemma 3.1} we only need to consider those integers $n$ whose prime factors are contained in $R_{\mathbb{Q}}(10)=\{ 2, 3, 5, 7, 11 \}$.

\begin{lemma}
\label{Lemma 6.1}
Let $E/K$ be an elliptic curve with rational j-invariant. Then $E(K)$ can't contain $C_{121}$, $C_{15}$, $C_{50}$ or $C_{125}$.
\end{lemma}

\begin{proof}
\fbox{$C_{121}$}: This is proven in \cite[Lemma 2.9.]{1}.
\\
\\
\fbox{$C_{15}$}: By Lemma \ref{Lemma 3.4} and Lemma \ref{Lemma 3.5}, $E'$ has a rational $15$-isogeny, so $j(E') \in \{-5^2/2, -5^2 \cdot 241^3/2^3, -5\cdot 29^3/2^5, 5 \cdot 211^3/2^{15} \}.$ Let $\langle P_{15} \rangle$ be the kernel of such a rational $15$-isogeny and $P_{15}$ a point of order $15$ on $E'$. Since $\Gal(\mathbb{Q}(E[15])/\mathbb{Q})$ acts on $\langle P_{15} \rangle$, we get that $[\mathbb{Q}(P_{15}):\mathbb{Q}]$ divides $|(\mathbb{Z}/15\mathbb{Z})^{\times}|=8$. On the other hand, $\mathbb{Q}(P_{15}) \subseteq L$ and since $[L:\mathbb{Q}]=10$, we have $[\mathbb{Q}(P_{15}):\mathbb{Q}]=2$. By \cite[Theorem 2.c)]{7}, the $\mathrm{LMFDB}$ label of $E'$ is $\href{http://www.lmfdb.org/EllipticCurve/Q/50b1/}{\mathrm{50.b3}}$, $\href{http://www.lmfdb.org/EllipticCurve/Q/50b2/}{\mathrm{50.b4}}$, $\href{http://www.lmfdb.org/EllipticCurve/Q/50a3/}{\mathrm{50.a2}}$ or $\href{http://www.lmfdb.org/EllipticCurve/Q/450b4/}{\mathrm{450.g4}}$. Let $E' \in \{ \href{http://www.lmfdb.org/EllipticCurve/Q/50b1/}{\mathrm{50.b3}}, \href{http://www.lmfdb.org/EllipticCurve/Q/50b2/}{\mathrm{50.b4}} \}.$ Then we have $E'(\mathbb{Q})[5]=C_5$ and since $C_5 \oplus C_5 \subseteq E(K)[5] \oplus E'(K)[5] \cong E'(L)[5]$, the Weil pairing implies that $\mathbb{Q}(\zeta_5) \subseteq L$ which is impossible. Assume now that $E' \in \{ \href{http://www.lmfdb.org/EllipticCurve/Q/50a3/}{\mathrm{50.a2}}, \href{http://www.lmfdb.org/EllipticCurve/Q/450b4/}{\mathrm{450.g4}} \}.$ In these cases we have that $E'(\mathbb{Q})[3]=C_3$. Applying the same reasoning as before, we conclude that $\mathbb{Q}(\zeta_3) \subseteq L$. On the other hand these two curves attain $15$-torsion point over $\mathbb{Q}(\sqrt{5})$ and $\mathbb{Q}(\sqrt{-15})$, respectively. Therefore we have $\mathbb{Q}(\sqrt{5}) \subseteq L$ or $\mathbb{Q}(\sqrt{-15}) \subseteq L$, but this is impossible since $\mathbb{Q}(\zeta_3) \subseteq L$ is a unique quadratic subextension of $L$, because $[L:\mathbb{Q}]=10$.
\\
\\
\fbox{$C_{50}$}: By Lemma \ref{Lemma 3.4} we have that $E'$ has rational $2$ and $5$-isogenies. If $G_{E'}(5) \subseteq C_{s}(5)$, then $E'$ has two independent rational $5$-isogenies. Denote by  $\langle P \rangle$ and $\langle Q \rangle$ the kernels of these isogenies. We have that $\langle P_2+P \rangle$ and $\langle Q \rangle$ are kernels of independent $10$ and $5$-isogenies, so $E'$ is isogenous over $\mathbb{Q}$ to a curve $E''/\mathbb{Q}$ with a rational $50$-isogeny, which is impossible by Theorem \ref{Theorem 2.3}. Therefore we conclude that $G_{E'}(5) \in \{ 5B.1.1, 5B.1.2, 5B.1.3, 5B.1.4, 5B.4.1, 5B.4.2, 5B \}.$ A calculation in \texttt{}{Magma} shows that if $E'$ obtains a point $P_{25}$ of order $25$ over a degree $5$ or $10$ extension of $\mathbb{Q}$, then it must have a rational $25$-isogeny. But since it has a rational $2$-isogeny as well, it must have a rational $50$-isogeny, which is impossible by Theorem \ref{Theorem 2.3}. 
\\
\\
\fbox{$C_{125}$}: This is proven in \cite[Proposition 2.7.]{1}.
\end{proof}
\begin{theorem}
Let $K$ be a number field such that $[K:\mathbb{Q}]=5$. If $E(K)$ contains a point of order $n > 1$, then $C_{n} \in \Phi_{\mathbb{Q}}(5)$.
\end{theorem}

\begin{proof}
Assume that $E(K)$ contains a point $P_n$ of order $n$. By Lemma \ref{Lemma 3.1}, the prime factors of $n$ can only be $2, 3, 5, 7, 11$. Therefore, write $n=2^{a}3^{b}5^{c}7^{d}11^{e}$, where $a,b,c,d,e \ge 0$. If $d \neq 0$, then by Lemma \ref{Lemma 3.6} $E$ is a base change of an elliptic curve defined over $\mathbb{Q}$, so $C_n \in \Phi_{\mathbb{Q}}(5)$. Assume that $d=0$. If $e \neq 0$, from \cite[Table 2]{2} we see that $E'$ has a rational $11$-isogeny. If one of  $a,b,c$ is not zero, $E'$ would have a rational $p$-isogeny, where $p \in \{ 2, 3, 5 \}$ so it would have a rational $11p$-isogeny which is impossible by Theorem \ref{Theorem 2.3}. Therefore we have $a=b=c=0$. If $e \ge 2$, this is impossible by Lemma \ref{Lemma 6.1}. We conclude that if $e \ge 1$, then $n=11$. Consider now the case when $d=e=0$. If $c \ge 1$, by Lemma \ref{Lemma 3.4} and Lemma \ref{Lemma 3.5} $E'$ has a rational $2^{a}3^{b}5$-isogeny. By Theorem \ref{Theorem 2.3} we have $2^{a}3^{b}5 \in \{ 5, 10, 15 \}$, so $(a,b) \in \{ (0,0), (1,0), (0,1) \}$ and $n \in \{ 5^c, 2 \cdot 5^c, 3 \cdot 5^c \}$. We have that this is impossible by Lemma \ref{Lemma 6.1} unless $n \in \{ 5, 10, 25 \}$, but for those values of $n$ we have $C_{n} \in \Phi_{\mathbb{Q}}(5)$. Finally let us consider the case when $n=2^{a}3^{b}$. By Theorem \ref{Theorem 2.3} we know that \[ n=2^{a}3^{b} \in \{ 1,2,3,4,6,8,9,12,16,18,27 \} .\] Removing the $n$ for which $C_n \in \Phi_{\mathbb{Q}}(5)$ we only need to consider $n \in \{ 16, 18, 27 \}$. These three cases have been proven to be impossible in Lemma \ref{Lemma 3.7}, Lemma \ref{Lemma 3.9} and Lemma \ref{Lemma 3.10}.
\end{proof}

\begin{lemma}
Let $E/K$ be an elliptic curve with rational j-invariant. Then $E(K)$ can't contain $C_{2} \oplus C_{10}$ or $C_{2} \oplus C_{12}$.
\end{lemma}
\begin{proof}
\fbox{$C_{2} \oplus C_{12}$}: This has been proven in Lemma \ref{Lemma 3.8}.
\\
\\
\fbox{$C_{2} \oplus C_{10}$}: Since $2$-torsion does not change when twisting, $E(K)[2]=C_2 \oplus C_2$ implies that $E'(K)[2]=C_2 \oplus C_2$. From $[K:\mathbb{Q}]=5$ it follows that $E'(\mathbb{Q})[2]=C_2 \oplus C_2$. Lemma \ref{Lemma 3.5} implies that $E'$ has a rational $5$-isogeny and we know that it has full $2$-torsion over $\mathbb{Q}$. It follows that $E'$ is isogenous over $\mathbb{Q}$ to $E''/\mathbb{Q}$ that has $20$-isogeny, which is impossible by Theorem \ref{Theorem 2.3}.
\end{proof}

\section{$[K:\mathbb{Q}]=3$}
In this subsection, let $K$ denote a number field such that $[K:\mathbb{Q}]=3$.
\\
If $P_n \in E(K)$ is a point of order $n$ then $C_n \subseteq E'(L)$. By Lemma \ref{Lemma 3.1} we only need to consider those integers $n$ whose prime factors are contained in $R_{\mathbb{Q}}(6)=\{ 2, 3, 5, 7, 13 \}$.
\begin{theorem}
Let $K$ be a number field such that $[K:\mathbb{Q}]=3$. If $E(K)$ contains a point of order $n > 1$, then $C_{n} \in \Phi_{\mathbb{Q}}(3)$.
\end{theorem}

\begin{proof}
Let $C_n \subseteq E(K)$. We have that $C_n  \subseteq E'(L)$ where $[L:\mathbb{Q}]=6$. Therefore, by Theorem \ref{Theorem 2.8} we have that $C_n$ is equal to the one of the following groups: 
\[C_m: m=1, \ldots, 10, 12, 13, 14, 15, 16, 18, 21, 30 .\] If $C_{n} \in \Phi_{\mathbb{Q}}(3)$, we are done. Assume that this isn't the case. Then $n \in \{ 15, 16, 30 \}$. Obviously it's enough to show that $n=15$ and $n=16$ is impossible. If $n=15$, by Lemma \ref{Lemma 3.5} we have that $E$ is a base change of an elliptic curve defined over $\mathbb{Q}$. Since $C_{15} \notin \Phi_{\mathbb{Q}}(3)$, we are done. The case when $n=16$ has been proven in Lemma \ref{Lemma 3.7}.
\end{proof}

\begin{lemma}
Let $E/K$ be an elliptic curve with rational j-invariant. Then $E(K)$ can't contain $C_{2} \oplus C_{10}$, $C_{2} \oplus C_{12}$ or $C_{2} \oplus C_{18}$.
\end{lemma}

\begin{proof}
\fbox{$C_{2} \oplus C_{10}$}: If $C_{2} \oplus C_{10} \subseteq E(K)$, by Lemma \ref{Lemma 3.5} and Lemma \ref{Lemma 3.6} we have that $E$ is a base change of an elliptic curve defined over $\mathbb{Q}$, so $C_{2} \oplus C_{10} \in \Phi_{\mathbb{Q}}(3)$, which contradicts Theorem \ref{Theorem 2.5}.
\\
\\
\fbox{$C_{2} \oplus C_{12}$}: This has been proven in Lemma \ref{Lemma 3.8}.
\\
\\
\fbox{$C_{2} \oplus C_{18}$}: Assume that $C_2 \oplus C_{18} \subseteq E(K)$. A point $P_3$ of order $3$ on $E'$ is defined over an at most quadratic extension of $\mathbb{Q}$, by \cite[Table 1]{1}. Using Lemma \ref{Lemma 3.2} we can assume that $E'(\mathbb{Q})[3]=C_3$. Since $C_3 \oplus C_3 =E(K)[3] \oplus E'(K)[3] \cong E'(L)[3]$, $C_9 \subseteq E'(L)$ and $E'(L)[2]=C_2 \oplus C_2$, it follows that $C_6 \oplus C_{18} \subseteq E'(L)$. Obviously $G_{E}(2) \ne 2B$. By Theorem \ref{Theorem 2.8}, this is impossible.

\end{proof}
\section{$[K:\mathbb{Q}]=2$}
In this subsection, let $K$ denote a number field such that $[K:\mathbb{Q}]=2$.
\\
\begin{theorem}
Let $K$ be a number field such that $[K:\mathbb{Q}]=2$. If $E(K)$ contains a point of order $n > 1$, then $C_{n} \in \Phi_{\mathbb{Q}}(2) \cup \{ C_{13} \}$.
\end{theorem}

\begin{proof}
Obviously we have $\Phi_{\mathbb{Q}}(2) \subseteq \Phi_{j \in \mathbb{Q}}(2) \subseteq \Phi(2)$. By Theorem \ref{Theorem 2.2} and Theorem \ref{Theorem 2.4} we have $\Phi(2) \setminus \Phi_{\mathbb{Q}}(2) =\{ 11, 13, 14, 18 \}$. Let $n$ be one of those values. Obviously $E/K$ such that $E(K)$ contains a point $P_n$ of order $n$ can't be a base change of an elliptic curve defined over $\mathbb{Q}$, by Theorem \ref{Theorem 2.4}. Therefore, let $E'/\mathbb{Q}$ be an elliptic curve such that $j(E)=j(E')$ and let $L$ be a quartic field over which $E$ and $E'$ are isomorphic. This implies that $E'$ has a point of order $n \in \{ 11, 14, 18\}$, but this is impossible since $C_n \notin \Phi_{\mathbb{Q}}(4)$ by Theorem \ref{Theorem 2.6}.
\\
Let now $E''/\mathbb{Q}$ be an elliptic curve and $L$ a quartic Galois extension of $\mathbb{Q}$ such that $P_{13} \in E''(L)$, where $P_{13}$ is a point of order $13$. Such an elliptic curve $E''$ and a field $L$ exist by \cite[Theorem 1.2]{22}. Denote by $K$ the intermediate field $\mathbb{Q} \subseteq K \subseteq L$. Since $L=K(\sqrt{d})$, consider a twist $E^{d}/K$ of $E''/K$ by $d$. We have 
\[ C_{13} \cong E''(L)[13]=E^{d}(K)[13] \cong E''(K)[13] .\] We conclude that $C_{13} \subseteq E^{d}(K)[13]$ and so $E^{d}/K$ is an elliptic curve with $j(E^d) \in \mathbb{Q}$ defined over a quadratic extension of $\mathbb{Q}$ with a point of order $13$.
\end{proof}
Remark: The sets $\Phi_{j \in \mathbb{Q}}(d)$ for $d=2,3$ were also determined by Manolis Tzortzakis in his thesis under the assumption that we know the set $\Phi(3)$. We did not assume the knowledge of $\Phi(3)$. The author has been working on determining the set $\Phi_{j \in \mathbb{Q}}(4)$ and has obtained partial results.

\section*{Acknowledgements}
The author would like to thank his advisor Filip Najman for introducing him to this problem, for many helpful discussions and for help during the writing of this paper.

\bibliographystyle{abbrv}

\end{document}